\documentclass[11pt]{article}
\usepackage{enumerate,amsthm,amssymb,amsthm,amsfonts,latexsym,color,graphicx}

\usepackage[left=1.1in, right=1in, top=1.1in, bottom=1.1in]{geometry}

\newtheorem{theorem}{Theorem}[section]
\newtheorem{lemma}[theorem]{Lemma}

\newtheorem{claim}[theorem]{Claim}
\newtheorem{definition}[theorem]{Definition}

\newcommand\conv{\mathrm{conv}}
\newcommand\eps\varepsilon
\def \reals{{\mathbb{R}}}
\def \A{{\mathcal{A}}}

\begin{document}

\begin{titlepage}

\title{Planar point sets determine many pairwise crossing segments}

\date{}

\author{J\'anos Pach\thanks{R\'enyi Institute, Budapest and MIPT, Moscow. Supported by NKFIH grants K-131529, KKP-133864, Austrian Science Fund Z 342-N31, Ministry of Education and Science of the Russian Federation MegaGrant No. 075-15-2019-1926, ERC Advanced Grant ``GeoScape.'' Email:
{\tt pach@cims.nyu.edu}.}
  \and Natan Rubin\thanks{Ben Gurion University of the Negev, Beer-Sheba, Israel. Ralph Selig Career Development Chair in Information Theory. The project leading to this application has received funding from European Research Council (ERC)
under the European Unions Horizon 2020 research and innovation programme under grant agreement No. 678765. Also supported
by grant 1452/15 from Israel Science Foundation. Email: {\tt rubinnat.ac@gmail.com}}
  \and G\'abor Tardos\thanks{R\'enyi Institute, Budapest and MIPT, Moscow. Supported by the ``Lend\"ulet'' Project in Cryptography of the Hungarian Academy of Sciences, the National Research, Development and Innovation Office NKFIH projects K-116769, SNN-117879, K-132696, KKP-133864 and SNN-13564, the ERC Advanced Grant ``GeoSpace'' and by the Ministry of Education and Science of the Russian Federation MegaGrant No. 075-15-2019-1926. Email: {\tt tardos@renyi.hu}}}

\clearpage\maketitle
\thispagestyle{empty}

\begin{abstract}
We show that any set of $n$ points in general position in the plane determines $n^{1-o(1)}$ pairwise crossing segments. The best previously known lower bound, $\Omega\left(\sqrt n\right)$, was proved more than 25 years ago by Aronov, Erd\H os, Goddard, Kleitman, Klugerman, Pach, and Schulman. Our proof is fully constructive, and extends to dense geometric graphs.
\end{abstract}

\end{titlepage}

\section{Introduction}
Let $V$ be a set of $n$ points in general position in the plane, that is, assume that no $3$ points of $V$ are collinear. A {\em geometric graph} is a graph $G=(V,E)$ whose vertex set is $V$ and whose edges are represented by possibly crossing straight-line segments connecting certain pairs of points in $V$. If every pair of points in $V$ is connected by a segment, we have $E={V\choose 2}$, and $G$ is called a {\em complete geometric graph}. Two edges $pq, p'q'\in E$ are said to {\em cross} if the corresponding segments share an interior point. {\em Topological graphs} are defined similarly, except that their edges can be represented by any Jordan curves that have no interior points that belong to $V$. In the present paper, we investigate crossing patterns of curves, mainly segments. The above general position assumptions will simplify the presentation, but our theorems hold without them. In particular, for our purposes, it is sufficient to consider topological graphs whose edges are non-selfintersecting polygonal curves.

\smallskip
\noindent{\bf Crossing patterns and intersection graphs.} Finding maximum cliques or independent sets in intersection graphs of segments, rays, and other convex sets in the plane is a computationally hard problem and a classic topic in computational and combinatorial geometry \cite{AgMu, CaCa, FP11, Chaya,KrM, KrNe}. There are many interesting Ramsey-type problems and results about the existence of large cliques {\em or} large independent sets in intersection graphs of segments  \cite{Semi, CPS, Ky, LMPT, PKK} and, more generally, of Jordan curves (``strings'') \cite{FP10,FPT11,FP12}. Some of these questions are intimately related to counting incidences between points and lines \cite{PachSharir, Szekely, SzT}, and to bounding the complexity of $k$-levels in arrangements of lines in $\reals^2$ \cite{Dey}.
\smallskip

It appears to be a somewhat simpler task to understand the {\it combinatorial structure} of crossings between the edges of a geometric or topological graph. Despite decades of steady progress, we have very few asymptotically tight results in this direction.
Perhaps the best known and most applicable theorem of this kind is the so-called {\em Crossing Lemma} of Ajtai, Chv\'atal, Newborn, Szemer\'edi~\cite{Crossing1} and Leighton~\cite{Crossing2}, which states that any topological graph $G=(V,E)$ with $|E|>4|V|$ determines at least $\Omega\left(|E|^3/|V|^2\right)$ crossing pairs of edges. Recently, a similar result has been established by the authors for contact graphs of families of Jordan curves \cite{PRTadv}.
\smallskip

According to another asymptotically tight result, for $t>1$, every geometric graph $G$ with $|E|\geq nt$ edges has two disjoint sets of edges, $E_1,E_2\subset E$, each of size $\Omega(t)$, such that every edge in $E_1$ crosses all edges in $E_2$; see, e.g., \cite[Theorem 6]{FP10}. A similar theorem holds for topological graphs, with the difference that then $|E_1|, |E_2|= \Omega\left(t/\log t\right)$ \cite{FP12}.
It is a major unsolved question to decide whether under these circumstances $G$ must also contain a family of {\em pairwise crossing} edges, whose size tends to infinity as $t\rightarrow\infty$; see~\cite{Book}, Chapter 9.6, Problem 1. It is conjectured that one can always choose such a family consisting of almost $t$ edges. If this stronger conjecture is true for $t^{1-o_t(1)}$ edges, then for $t\approx n/2$, it would imply that every {\em complete} geometric or topological graph on $n$ vertices has $n^{1-o(1)}$ pairwise crossing edges (cf.~\cite{Book}, Chapter 9.6, Problem 4).

A topological graph is called {\em $t$-quasi-planar} if it contains no $t$ pairwise crossing edges. Let $f_t(n)$ (and $f'_t(n)$) denote the maximum number of edges that a $t$-quasi-planar geometric (resp., topological) graph of $n$ vertices can have. Clearly, we have $f_t(n)\le f'_t(n)$, for every $t$ and $n$. For geometric graphs, Valtr~\cite{Va3} proved that $f_t(n)=O_t(n\log n)$, but in general the best known upper bound is only $f'_t(n)=O_t\left(n(\log n)^{O(\log t)}\right)$ \cite{FP10,FP12a,FP14}. It is conjectured that $f_t(n)\le f'_t(n)\le t^{1+o_t(1)}n$, which is known to be true only for $t\le 4$; see \cite{Eyal2, Eyal1, AAPP}.

\smallskip
\noindent{\bf Our results.}  The aim of the present paper is to find many pairwise crossing edges in dense geometric graphs. For every $n\geq 2$, let $T(n)$ denote the largest positive integer $T$ with the property that any complete geometric graph with $n$ vertices has at least $T$ pairwise crossing edges. Equivalently, $T(n)$ is the largest number such that any set $V$ of $n$ points in general position in the plane determines at least $T$ pairwise crossing segments. (A segment is {\em determined} by $V$ if both of its endpoints belong to $V$.)

It was proved by Aronov, Erd\H{o}s, Goddard, Kleitman, Klugerman, Pach, and Schulman \cite{CrossingFamilies} ipn 1991 that $T(n)=\Omega(\sqrt{n})$, cf. \cite{Va2}. Since then no one has been able to improve this bound. The prevailing conjecture is that $T(n)=\Theta(n)$ \cite{Book}.
A slightly sublinear lower bound holds for ``uniformly distributed'' point sets, in which the ratio of the largest distance and the smallest distance between two points is $O(\sqrt{n})$ \cite{Va1}. Note that a very recent construction by Evans and Saeedi \cite{EvSa} yields $T(n)\leq 5\lceil n/24 \rceil$, which was lately improved to $T(n)\le \lceil n/5\rceil$ by a group of Czech and Austrian researchers \cite{Ai}.

Our main theorem comes close to settling the above conjecture in the affirmative, and applies in a more general setting.

\begin{theorem}\label{Thm:Main}
(i) Any set $V$ of $n$ points in general position in the plane determines at least $n/2^{O(\sqrt{\log n})}$ pairwise crossing segments.

(ii) There exists an absolute constant $c$ such that any geometric graph with $n$ vertices and at least $n^{2-\eps}$ edges with $\eps>(\log n)^{-2/3}$ has at least $n^{1-c\sqrt\eps}$ pairwise crossing edges.
\end{theorem}

Ignoring the specific order of the error term, part (i) of the theorem implies the existence of $n^{1-o(1)}$ pairwise crossing edges in a {\it complete} geometric graph with $n$ vertices, while part (ii) implies the same in {\em all} geometric graphs with $n$ vertices and $n^{2-o(1)}$ edges. Both parts of the theorem will follow from Lemma~\ref{4+} presented in Section~\ref{3rd}, which is the heart of our argument.

\smallskip
Our proof of Theorem \ref{Thm:Main} is fully constructive. 
The crossing segments can be found by an efficient algorithm whose running time is near-quadratic in $n$ for complete or dense geometric graphs. Our construction heavily relies on the assumption that the segments are straight, but otherwise it is fairly robust. Below we briefly sketch the main ideas of our proof of Theorem \ref{Thm:Main}.

\begin{figure}[htbp]
\begin{center}
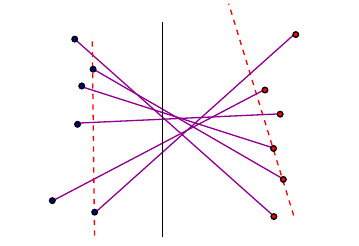
\caption{\small An avoiding pair $\{A,B\}$ with $|A|=|B|=6$, and its induced family of $6$ pairwise crossing edges. The points of $B$ see the points of $A$ in the same counter-clockwise order $p_1,p_2,p_3,p_4,p_5,p_6$, and a similar property holds for the points of $A$. Equivalently, the sets $A$ and $B$ are separated by a line, and all the points of $B$ lie to the same side of each line $p_ip_j$, for $1\leq i\neq j\leq 6$, and vice versa; see Subsection \ref{Posets}.}
\label{Fig:Avoidance}
\end{center}
\end{figure}

The main observation of
Aronov et al. \cite{CrossingFamilies} was that any $n$-point set $V$ in general position in the plane contains a pair of subsets, $A$ and $B$, each of cardinality $|A|=|B|=\Omega\left(\sqrt{n}\right)$ that can be separated by a line, and so that all the points in $B$ see the points in $A$ in the same order, and vice versa. Such a pair $A,B\subset V$ is called {\it avoiding}, and its properties are reviewed in Subsection \ref{Posets}. Connecting each point of $A$ to the ``opposite'' point of $B$, we obtain a family of $m=\Omega\left(\sqrt{n}\right)$ pairwise crossing segments; see Figure \ref{Fig:Avoidance}. Valtr \cite{Va2} showed that, for certain instances of $V$, one cannot find an avoiding pair with $m$ greater than some constant times $\sqrt{n}$.

The key insight that allows us to get around this barrier is that we do not necessarily have to find two {\em avoiding} point sets in order to complete our proof. We can still proceed, albeit in a somewhat trickier manner, if we can find two $m$-element sets, $A$ and $B$, with $m$ close to $n$, which are ``nearly avoiding'' in the following sense. The points of $A$ are separated from the points of $B$ by a line, and all points of $B$ see the points in $A$ in the same ``approximate'' order and {\em vice versa}. More precisely, we will introduce two partial orders  $<_B$ on $A$ and $<_A$ on $B$ with only $o\left(m^2\right)$ incomparable pairs. To obtain such a nearly avoiding pair $(A,B)$ with $|A|=|B|=\Omega\left(n^{1-o(1)}\right)$, we use the $\epsilon$-net machinery from computational geometry  \cite{M}.

\begin{figure}[htbp]
\begin{center}
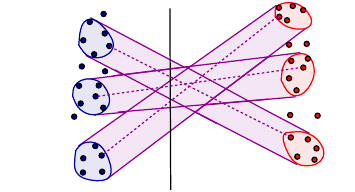
\caption{\small The sets $A$ and $B$ are nearly avoiding. In the depicted scenario, most points in $A$ (resp., $B$) are partitioned into smaller sets $\{C_i\}_{i=1}^3$ (resp., $\{D_i\}_{i=1}^3$). A matching of pairwise crossing ``super-edges'' $C_1\times D_1, C_2\times D_2, C_3\times D_3$ is depicted.}
\label{Fig:SuperEdges}
\end{center}
\end{figure}

Armed with our weaker variant of avoidance, we further partition most points in $A$ (resp., $B$) into smaller subfamilies $\{C_i\}_{i=1}^k$ (resp., $\{D_i\}_{i=1}^k$) so that any two vertices of $\bigcup_{i=1}^k C_i$ that are incomparable under $(A,<_B)$ lie in the same set $C_i$, and a symmetric property holds for $\bigcup_{i=1}^k D_i$. It is easy to check that for any four distinct sets $C_{i},C_{j}, D_{i'}$ and $D_{j'}$, the property that an edge $xy\in C_i\times D_{i'}$ crosses another edge $zw\in C_{j}\times D_{j'}$ is independent of the choice of representatives $x,y,z,w$.
Finally, we use Mirsky's theorem (the dual of Dilworth's Theorem \cite{Dilworth}) to obtain a large family of pairwise crossing such ``super-edges'' $C_i\times D_{i'}$ while also maintaining the ``nearly avoiding'' property of these pairs. (See Figure \ref{Fig:SuperEdges}.)

\medskip
\noindent{\bf Paper organization.} The rest of the paper is organized as follows.
In Section~\ref{lemmas}, we collect the essential facts needed for the proof of Theorem \ref{Thm:Main}. Specifically, in Subsection \ref{Posets} we establish two lemmas concerning partially ordered sets and their application to families of pairwise crossing segments. In Subsection \ref{Subsec:Zone}, we present a partition result (implicitly shown by Matou\v{s}ek \cite{M}) which will be instrumental in finding large point sets that satisfy our approximate version of avoidance.

The proof of Theorem~\ref{Thm:Main} is given in Section~\ref{3rd}.
We first obtain a pair of nearly avoiding sets $A,B\subset V$, each of cardinality $\Omega\left(n^{1-o(1)}\right)$, and then we show that they induce a large family of pairwise crossing edges.

Section \ref{Sec:Discuss} contains some remarks, including a brief discussion of the constructive aspects of Theorem \ref{Thm:Main}. We also provide an analogue of Theorem \ref{Thm:Main} which yields many pairwise evading edges in dense geometric graphs. Here, two edges are said to be {\it evading} if the supporting line of neither of them intersects the closure of the other  \cite{CrossingFamilies,Pinchasi,Va3}. In particular, evading edges must be disjoint.

\section{Preliminaries}\label{lemmas}

\subsection{\bf Partially ordered sets}\label{Posets}
We use the term \emph{poset} for finite partially ordered sets $(P,<)$. If the ordering is clear from the context, we write $x||y$ to indicate that the elements $x,y\in P$ are {\em incomparable}. Let $\iota(P,<)$ denote the number of incomparable pairs of elements in $P$. For subsets $A$ and $B$ of $P$ we write $A<B$ if $a<b$ for all $a\in A$ and $b\in B$.

We start with the following lemma which is somewhat similar to Theorem 4 (ii) in \cite T.

\begin{lemma}\label3
Let $n$ and $k$ be positive integers, and $(P,<)$ be a poset with $|P|>nk$ and $\displaystyle \iota(P,<)\le (|P|-nk)^2/(16k)$.

\medskip
Then  one can choose suitable $n$-element subsets $A_1,A_2,\dots,A_k$ of $P$ with $A_i<A_j$ for all $i<j$.
\end{lemma}

\begin{proof}
Let $I_x=\{y\in P\mid y||x\}$ be the set of elements incomparable to $x\in P$. Let $T=(|P|-nk)/(4k)$ and $Q=\{x\in P\mid |I_x|<T\}$. Since $\sum_{x\in P}|I_x|=2\iota(P,<)$, we have
$$|P\setminus Q|\le\frac{2\iota(P,<)}{T}\le\frac{|P|-nk}{2}.$$

Let $x_1,x_2,\dots,x_m$ be a linear ordering of the elements of $Q$ compatible with the partial order $<$. That is, $m=|Q|$, $Q=\{x_1,\dots,x_m\}$, and whenever $x_i<x_j$, we have $i<j$. If $x_i||x_j$ for some $1\le i<j\le m$, then each element $x_l$ with $i\le l\le j$ must be incomparable to either $x_i$ or $x_j$. As both $x_i$ and $x_j$ are incomparable with less than $T$ elements, this implies $j-i<2T$. Therefore $j-i\ge2T$ implies $x_i<x_j$.

We choose the sets $A_i$ to be intervals of the linear order on $Q$ with enough buffer between them to ensure the comparability. Namely, we set $A_i=\{x_j\mid (i-1)(n+\lfloor2T\rfloor)<j\le(i-1)(n+\lfloor2T\rfloor)+n\}$ for every $1 \le i \le k$.  We only need to check that $Q$ is large enough for all these intervals to fit. This holds, because
$$m= |Q|=|P|-|P\setminus Q|\ge |P| -\frac{|P|-nk}{2}\ge k(n+2T).$$
\end{proof}

Let $x$ and $y$ be distinct points in the plane. Let us orient the line $xy$ from $x$ toward $y$, and denote the open  half-plane to the left of this oriented line by $\ell(xy)$. Let $B$ be a nonempty set in the plane. For any pair of distinct points in the plane, $x$ and $y$, write $x<_By$ if $B$ is contained in $\ell(xy)$; see Figure \ref{Fig:PartialOrder} (left). The convex hull of a set $B$ is denoted by $\conv(B)$.

\begin{figure}[htbp]
\begin{center}
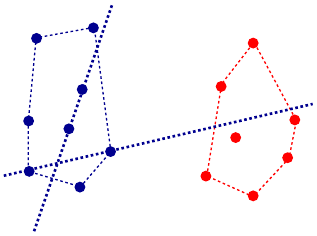\hspace{3cm}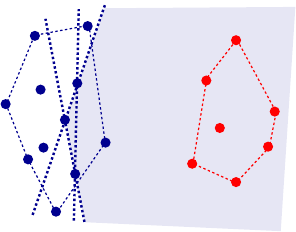
\caption{\small  Left: The partial order $(A,<_B)$. We have $x<_By$ if and only if $B$ lies in the half-plane $\ell(xy)$ to the left of the directed line from $x$ to $y$. (The points $x'$ and $y'$ are incomparable, because the line $xy$ crosses the convex hull of $B$.) Right: Lemma \ref{4} -- proving that the relation $(A,<_B)$ is transitive.}
\label{Fig:PartialOrder}
\end{center}
\end{figure}


\begin{lemma}\label{4}
Let $A$ and $B$ be nonempty sets in the plane such that their convex hulls are disjoint.
Then $<_B$ defines a partial order on $A$, in which two distinct points $x$ and $y$ are incomparable if and only if the line through $x$ and $y$ intersects $\conv(B)$. Further, if $x<_By$ for $x,y\in A$ and $z<_At$ for $z,t\in B$, then the segments $xz$ and $yt$ cross.
\end{lemma}

\begin{proof}
To check transitivity, let us assume that $x<_By$ and $y<_Bz$ for a triple of points $x,y,z\in A$; see Figure \ref{Fig:PartialOrder} (right). This means that both $\ell(xy)$ and $\ell(yz)$ contain $B$. We have $$\ell(xz)\supseteq(\ell(xy)\cap\ell(yz))\setminus\conv(xyz),$$ and $\conv(xyz)$ is disjoint from $B$.  Therefore, $B\subseteq\ell(xz)$ must hold, which means that $x<_Bz$, as required.

The distinct points $x$ and $y$ are incomparable in $<_B$ if and only if $B$ is not contained in either of the open half-planes bounded by the line $xy$, which happens if and only if the line intersects the convex hull of $B$.

Finally, for any $x,y\in A$ and $z,t\in B$ that satisfy the inequalities $x<_By$ and $z<_At$, the quadrilateral $xyzt$ must be strictly convex, so the diagonals $xz$ and $yt$ must cross.
\end{proof}


\medskip
\noindent{\bf Definition.} A {\it chain} in a poset $(P,<)$ is defined as a sequence $x_1<x_2<x_3<\ldots <x_k$, so that $x_i\in P$ for all $1\leq i\leq k$. Accordingly, an {\it antichain} in a poset $(P,<)$ is a subset $Q\subset P$ of elements no two of which are comparable.

\begin{theorem}[Mirsky \cite{Dilworth}]\label{Thm:DualDilworth}
Let $(P,<)$ be a finite poset. Then the length of the largest chain in $(P,<)$ is equal to the size $k$ of its smallest possible decomposition $P=Q_1\uplus Q_2\uplus\ldots\uplus Q_k$ into antichains $Q_i\subseteq P$, for $1\leq i\leq k$.
\end{theorem}

\noindent{\bf From partial orders to pairwise crossing segments.}
Two finite point sets $A$ and $B$ in $\reals^2$ are called {\it avoiding} \cite{CrossingFamilies} if $\conv(A)\cap \conv(B)=\emptyset$, and each of the induced partial orders $(A,<_B)$ and $(B,<_A)$ (defined in the paragraph above Lemma \ref{4}) is a {\em total} order; that is, we have $\iota(A,<_B)=0$ and $\iota(B,<_A)=0$. Equivalently, each point in $B$ sees the points of $A$ in exactly the same counter-clockwise order $(A,<_B)$, and vice versa.

Aronov {\it et al.} \cite{CrossingFamilies} established one-to-one correspondence between the elements of any avoiding pair $\{A,B\}$, for $|A|=|B|=m$, such that any two segments formed by the corresponding points cross each other (as is illustrated on the right-hand side of Figure \ref{Fig:Avoiding}).
Indeed, as  $<_B$ linearly orders $A$ and vice versa, we have that $A=\{x_1,x_2,\dots,x_t\}$ with $x_1<_Bx_2<_B\cdots<_Bx_t$, and, analogously, $B=\{y_1,y_2,\ldots,y_t\}$ with $y_1<_Ay_2<_A\cdots<_Ay_t$.  By Lemma~\ref4, the segments $x_iy_i$ are pairwise crossing, for $i=1,\ldots,t$.
Furthermore, they showed that any $n$-element point set $V\subset \reals^2$ contains a pair of avoiding subsets $A,B\subset V$ with $m=|A|=|B|=\Omega\left(\sqrt{n}\right)$. Hence, every set of $n$ points in general position in the plane determines $m=\Omega(\sqrt{n})$ pairwise crossing segments.
As was mentioned in the Introduction, this approach cannot yield any better result, because there are $n$-element point sets in which the smaller of every pair of avoiding subsets has size $O\left(\sqrt{n}\right)$.

However, avoidance is not a necessary condition for the existence of large families of pairwise crossing segments between the two sets, $A$ and $B$. In Section~\ref{3rd}, we relax this condition. We introduce the notion of ``nearly avoiding'' pairs of sets, and, using the partition result introduced in the next subsection, we show that every $n$-element point set in the plane contains two almost linear-sized subsets, $A$ and $B$, that form a nearly avoiding pair. We complete the proof of Theorem~\ref{Thm:Main} by selecting close to $|A|$ pairwise crossing edges from the set of all segments connecting $A$ and $B$.

\subsection{Line arrangements} \label{Subsec:Zone}
Any finite family $L$ of lines in $\reals^2$ induces the {\it arrangement} $\A(L)$ -- the partition of $\reals^2\setminus \left(\bigcup L\right)$ into $2$-dimensional {\it cells}, or {\it $2$-faces}.
Each of these cells is a maximal connected region of $\reals^2\setminus \left(\bigcup L\right)$; it is a (possibly unbounded) convex polygon whose boundary is composed of {\it edges} -- portions of the lines of $L$, which connect vertices -- crossings amongst the lines of $L$.

\begin{figure}[htb]
\begin{center}
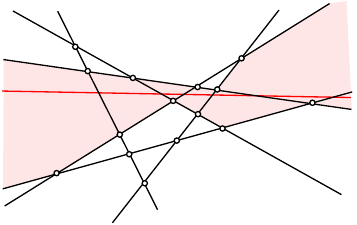
\caption{\small An arrangement of $6$ lines. The cells in the zone of a (seventh) line $\ell$ are shaded.}
\label{Fig:Zone}
\end{center}
\end{figure}

The \emph{zone} of a line $\ell$ within the arrangement $\A(L)$ is the union of all the (open) cells that $\ell$ intersects; see Figure \ref{Fig:Zone}.
The following result was implicitly established by Matou\v{s}ek \cite{M}; it comprises most of the proof of the so called Test Set Lemma.


\begin{lemma}[Matou\v sek, \cite M]\label5
For any $n$-element point set $V$ in general position in the plane and for any $\eps>0$, we can find $O(1/\eps^{2})$ lines such that the zone of any other line within their arrangement contains at most $\eps n$ points of $V$.
\end{lemma}

For his proof, Matou\v sek \cite{M} used a result of Chazelle and Friedman \cite{CF}, according to which any set of $n$ lines admits an {\it optimal $\eps$-cutting}. In other words, one can partition the plane into $O\left(\eps^{-2}\right)$ simply shaped cells (triangles or trapezoids, some of which are unbounded) with the property that every cell is crossed by at most $\eps n$ lines. He applied this result to the dual line set  $V^*$ of $V$ \cite[Section 5]{Matousek}, and argued that the duals of the vertices of this cutting (that correspond to lines in the primal plane) satisfy the requirements of Lemma \ref{5}. Before giving a formal proof of the lemma, we outline a more explicit argument which yields the same result with the slightly weaker bound $O(\log^2(1/\eps)/\eps^2)$. This bound also suffices for our calculations.

We consider all {\it angular sectors} in $\reals^2$;  each of them can be obtained as the intersection of two open half-planes. We construct a set of points $Q\subset V$ which pierces every angular sector that contains at least $\eps n/4$ points of $V$; such a set is known as an {\it $(\eps/4)$-net}.
In accordance with the standard theory of (strong) $\eps$-nets \cite{HW}, angular sectors have a {\it bounded VC-dimension} (namely, $5$), so we can find such a net $Q$ with $|Q|=O\left(\log(1/\eps)/\eps\right)$. Note that the family $\mathcal L$ of all lines determined by the point set $Q$ is of size $O(\log^2(1/\eps)/\eps^2)$. To see that this family satisfies the requirements of the lemma, it suffices to check that the zone of any line $\ell$ is contained in the union of at most {\it four} angular sectors, each disjoint from $Q$. We show that the part of the zone in a half-plane $F$ bounded by $\ell$ can be covered by two such angular sectors. For simplicity, we assume $\ell$ does not pass through any point in $Q$ and is not parallel to a line determined by two points in $Q$. These restrictions are not essential but allow us to avoid a case analysis. If $Q\cap F$ is empty or consists of collinear points, then the statement is trivial. Otherwise, let $x$ be the closest point to $\ell$ in $Q\cap F$ and let $A$ be the smallest angular sector with apex $x$, whose closure contains $Q\cap F$; see Figure \ref{Fig:Angular}. Notice that the two delimiting lines of $A$ belong to the family $\mathcal L$. It readily follows that $F\setminus \overline A$ is disjoint from $Q$, it is the union of two angular sectors, and it contains the part of the zone of $\ell$ inside $F$.

\begin{figure}[htb]
\begin{center}
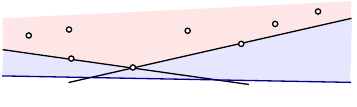
\caption{\small Proof of the relaxed variant of Lemma \ref{5}. The point $x\in Q$ is the closest point to $\ell$ in the half-plane $F$ (above $\ell$).  $A$ is the smallest angular sector with apex $x$, and whose closure contains $Q\cap F$. The complement of the closure of $A$ within $F$ is covered by two angular sectors which together cover the part of the zone of $\ell$ within $F$.}
\label{Fig:Angular}
\end{center}
\end{figure}

\begin{proof}[Proof of Lemma \ref{5}]
Let $V^*$ denote the set of lines dual to the points of $V$. Since the points of $V$ are in a general position, an analogous condition must hold for the dual set $V^*$: no three lines in $V^*$ can meet at the same point. Furthermore, a generic rotation of the orthogonal axes guarantees that none of these lines can be horizontal or vertical. We construct an optimal $\eps$-cutting\, $\Xi(V^*)$ of the dual plane with respect to $V^*$  \cite{CF}.
Namely, $\Xi(V^*)$ is a decomposition of the dual plane into $O\left(1/\eps^2\right)$ interior-disjoint triangles with the property that each of these triangles is crossed by at most $\lfloor \eps |V^*|\rfloor=\lfloor \eps n\rfloor$ dual lines belonging to $V^*$.\footnote{Some of the vertices of these triangles may lie at infinity on one of the imaginary lines $x=\infty$, $x=-\infty$, $y=\infty$, $y=-\infty$, which are incorporated in $V^*$.}

Our set, which we denote by $\mathcal L$, consists of the lines dual to the vertices of $\Xi(V^*)$. To see that $\mathcal L$ satisfies the asserted property, let $\ell$ be an arbitrary non-vertical line in the primal plane.
Assume first that the dual point $\ell^*$ is contained in the relative interior of some triangle $\Delta\in \Xi(V^*)$. Since the lines in $V^*$ are in a general position, the argument readily extends to the points that lie on boundaries of the triangles of $\Delta\in \Xi(V^*)$. The crucial observation is that for any point of $V$ that lies in the zone of $\ell$ within the arrangement of $\mathcal L$, its dual line $p^*\in V^*$ must cross the boundary of the triangle $\Delta$ (or, else, $p$ will be ``separated'' from $\ell$ by the lines dual to the vertices of $\Delta$); see \cite{Quasi} for the precise details. Hence, the number of such points cannot exceed $\eps n$.
\end{proof}

\section{Proof of Theorem \ref{Thm:Main}}\label{3rd}

We begin by introducing more terminology.

\begin{definition}\label{avoiding}
Two point sets $A, B\subset \mathbb{R}^2$ are called \emph{separated} if $\conv(A)\cap\conv(B)=\emptyset$. Two separated $m$-element sets, $A$ and $B$, are said to form an \emph{$\eps$-avoiding pair} for some $\eps\ge0$ if  $$\iota(A,<_B)+\iota(B,<_A)\le\eps m^2.$$ For simplicity, a $0$-avoiding pair is called \emph{avoiding}. See Figure~\ref{Fig:Avoiding} (left).
\end{definition}

For $\eps=0$, this coincides with the notion of avoidance introduced by Aronov {\it et al.} \cite{CrossingFamilies}.

\begin{figure}[htbp]
\begin{center}
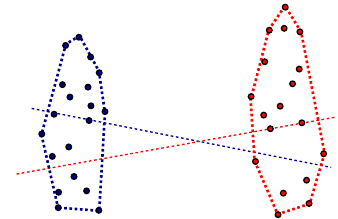 \hspace{3cm}
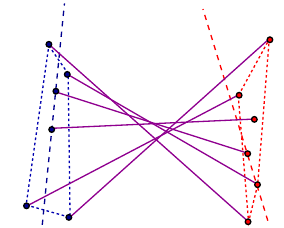
\caption{\small Left: Two separated sets $A$ and $B$ form an {\it $\eps$-avoiding pair} if there exist at most $\eps m^2$ incomparable pairs in ${A\choose 2}\cup {B\choose 2}$, whose connecting lines cross the convex hull of the opposite set. Right: A $0$-avoiding pair $\{A,B\}$ with $|A|=|B|=6$, and its induced family of $6$ pairwise crossing edges.}
\label{Fig:Avoiding}
\end{center}
\end{figure}

\medskip
To be able to deal with geometric graphs that are not complete, we introduce a further definition. Let $G=(V,E)$ be a geometric graph. For any subsets $A,B\subseteq V$, let $E(A,B)\subseteq E$ denote the set of edges connecting a vertex in $A$ with a vertex in $B$.

\begin{definition}
Given a geometric graph $G=(V,E)$ and a number $\delta\ge 0$, a pair $\{A,B\}$ of disjoint $m$-element subsets of $V$ is \emph{$\delta$-dense} with respect to $G$ if $|E(A,B)|\ge\delta m^2$.
\end{definition}

If it is clear what the underlying geometric graph is, with no danger of confusion we simply say that the pair is \emph{$\delta$-dense}. Our next result states that in every sufficiently large and dense geometric graph, one can find two separated $m$-element sets that form an $\eps$-avoiding $\delta$-dense pair.

\begin{lemma}\label6
There exists an absolute constant $c>0$ with the following property. For any integer $m>0$ and any reals $\eps,\delta>0$, every geometric graph $G=(V,E)$ with $|V|\ge \frac{cm}{\eps^{4}\delta^{5}}$ and $|E|\ge\delta|V|^2$ has two separated $m$-element sets of vertices that form an $\eps$-avoiding $\delta$-dense pair with respect to $G$.
\end{lemma}

\begin{proof}
Let us set $n=|V|$. We use Lemma~\ref5 to obtain an arrangement $\A(L)$ of a set $L$ of $r=O\left(1/(\eps\delta)^2\right)$ lines such that the zone of any line contains at most $\eps\delta n/2$ points of $V$. We assume that $r\ge3$.

Using parallel segments or half-lines that do not pass through any point of $V$, split each cell of $\A(L)$ into smaller cells such that all but at most one of them contain precisely $m$ elements of $V$, and if there is an exceptional cell, it contains fewer than $m$ points. Denote the resulting cell decomposition of $\reals^2$ by $\Pi$. Obviously, every cell in $\Pi$ is convex. The set of $m$ points of $V$ inside a non-exceptional cell of $\Pi$ is called a \emph{cluster}; see Figure \ref{Fig:Clusters}. Denoting the number of clusters by $D$, we have $D\le n/m$.

\begin{figure}[htb]
\begin{center}
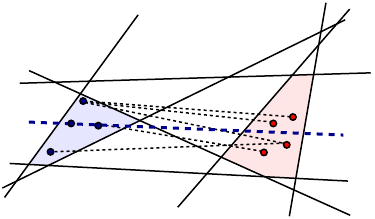
\caption{\small A pair of cells in the decomposition $\Pi$, which determine clusters $A$ and $B$. The edges of $E(A,B)$ are depicted. If the pair $x,y\in A$ is incomparable with respect to $<_B$, then $B$ must lie in the zone of the line $xy$.}
\label{Fig:Clusters}
\end{center}
\end{figure}

Let $H$ denote the set of points in $V$ that do not belong to any cluster. Such a point must lie either in an exceptional cell of $\Pi$ or on a line of $L$. There are at most $r^2-1$ exceptional cells, each containing fewer than $m$ points in $V$, and $\Pi$ has $r$ lines, each passing through at most $2$ points in $V$. Thus, we have $|H|\le (r^2-1)m$. Every edge of $G$ belongs to one of the following categories:

\smallskip
(i) it has a vertex belonging to $H$, or

(ii) it connects a pair of vertices in the same cluster, or

(iii) it connects a pair of vertices in distinct clusters that do {\em not} form a $\delta$-dense pair, or

(iv) it connects a pair of vertices in distinct clusters that form a $\delta$-dense pair.

\smallskip
The number of edges in category (i) is at most $|H|n\le(r^2-1)mn$. The number of edges in category (ii) is at most $D{m\choose2}<mn$. The number of edges in category (iii) is less than ${D\choose2}\delta m^2<\delta n^2/2$. By assumption, $G$ has at least $\delta n^2$ edges, so more than $\delta n^2/2-r^2mn$ of them must belong to category (iv). The number of edges between any two distinct clusters is at most $m^2$. Therefore, the number of unordered $\delta$-dense pairs of clusters is at least
\begin{equation}\label{eqn1}
\frac{\delta n^2/2-r^2mn}{m^2}=\frac{\delta n^2}{2m^2}-\frac{r^2n}{m}.
\end{equation}

Let $$X=\sum_{\{A, B\}}(\iota(A,<_B)+\iota(B,<_A)),$$
where the sum is taken over all unordered pairs of distinct clusters $\{A, B\}$.

We can estimate $X$ according to the point pairs incomparable with respect to a cluster. An unordered pair of distinct vertices $\{x,y\}$ will contribute to the sum $X$ only if $x$ and $y$ come from the same cluster, and in this case its contribution will be the number of other clusters whose convex hulls are crossed by the line $xy$. All of these clusters belong to the zone of the line $xy$ in the arrangement $\A(L)$. By the choice of $L$, the zone of any line contains at most $\eps\delta n/2$ vertices. Thus, the contribution of any pair of points is at most $\eps\delta n/(2m)$, so we have $$X\le D{m\choose2}\frac{\eps\delta n}{2m}<\eps\delta n^2/4.$$

On the other hand, each pair of clusters which is \emph{not} $\eps$-avoiding contributes more than $\eps m^2$ to $X$, so we have fewer than $X/(\eps m^2)$ such pairs. This implies that the number of not $\eps$-avoiding pairs of clusters is less than
\begin{equation}\label{eqn2}
\frac{\eps\delta n^2/4}{\eps m^2}=\frac{\delta n^2}{4m^2}.
\end{equation}

Clearly, each cluster has $m$ elements and any two of them are separated. We are done if we find an $\eps$-avoiding $\delta$-dense pair formed by two clusters. For this, it is sufficient to show that there are more $\delta$-dense pairs of clusters than pairs that are not $\eps$-avoiding. Comparing (\ref{eqn1}) and (\ref{eqn2}), this is true as long as
$$\frac{\delta n^2}{2m^2}-\frac{r^2n}{m}\ge \frac{\delta n^2}{4m^2},$$
that is, if $n\ge4mr^2/\delta$. This can be ensured by choosing the constant $c$ large enough as a function of the constant hidden in $r=O(1/(\eps\delta)^{2})$, which comes from Lemma~\ref5.
\end{proof}

The heart of the proof is the following partition result.

\begin{lemma}\label{4+}
Let $k$, $m$, and $t\ge3$ be positive integers, and set $\delta=1/t$, $\eps=1/(32t^2k)$. Let $A$ and $B$ be two separated sets of vertices in a geometric graph $G$ with $|A|=|B|=(t+1)km$. Suppose that $\{A, B\}$ is a $8\delta$-dense, $\eps\delta$-avoiding pair.

Then we can find pairwise disjoint $m$-element subsets $A_1, A_2,\ldots, A_k$ $\subset A$, $B_1, B_2, \ldots, B_k$ $\subset B$ such that for every $1\le i\le k$, $\{A_i, B_i\}$ is a $\delta$-dense, $\eps$-avoiding pair and, for $1\le i<j\le k$, all edges in $E(A_i,B_i)$ cross every edge in $E(A_j,B_j)$.
\end{lemma}

\begin{proof}
It follows from the assumption that $\{A, B\}$ is an $\eps\delta$-avoiding pair that the posets $(A, <_B)$ and $(B, <_A)$ both satisfy the conditions of Lemma~\ref3 with $n=m$ and with $tk$ in place of $k$. Indeed, by Definition~\ref{avoiding}, we obtain
$$\iota(A,<_B)+\iota(B,<_A)\le\eps\delta (t+1)^2k^2m^2=\frac{(1+1/t)^2km^2}{32t}<
\frac{(|A|-mtk)^2}{16tk}.$$
Hence, we can apply Lemma~\ref3 to find suitable subsets $C_i\subseteq A$ and $D_i\subseteq B$ for $1\le i\le tk$ such that $|C_i|=|D_i|=m$ for all $i$, and $C_i<_BC_j$, $D_i<_AD_j$ for all $i<j$; see Figure \ref{Fig:Decomposition} (left).

Lemma~\ref4 implies that all edges in $E(C_i,D_i)$ cross every edge in $E(C_j,D_j)$ as long as $i\ne j$, which suggests to select the pairs $(A_i,B_i)$ from among the the pairs $(C_i,D_i)$. Unfortunately, many of these pairs $(C_i,D_i)$ might \emph{not} be $\delta$-dense or $\eps$-avoiding. We must be careful when we select a suitable pair for a set $C_a$. Any collection of $\delta$-dense, $\eps$-avoiding pairs $\{(C_a,D_b)\}$ will work here (still by Lemma~\ref4), as long as $b=b(a)$ is a monotone increasing function of $a$.

\begin{figure}[htbp]
\begin{center}
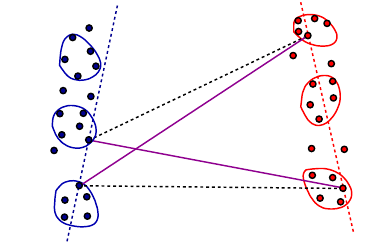
\hspace{2.5cm}
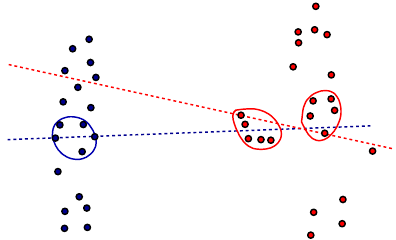
\caption{\small Left: The subsets $\{C_i\}$ and $\{D_i\}$. We have $C_i<_BC_j$ and $D_{i'}<_AD_{j'}$, for every $i<j$ and $i'<j'$. Notice that in this case,  by Lemma~\ref4, for any choice of $x\in C_i$, $z\in C_j$, $y\in D_{i'}$, and $w\in D_{j'}$, the segments $xy$ and $zw$ cross. Right: Proof of Claim~\ref{c1}. If $x$ and $y$ are incomparable with respect to both sets $D_a$ and $D_b$, and $a<b$, then $y$ lies to the right of the line from $z$ to $w$. This contradicts the assumption $D_a<_AD_b$.}
\label{Fig:Decomposition}
\end{center}
\end{figure}

Since $\{A,B\}$ is an $8\delta$-dense pair and $|A|=|B|=(t+1)km$, by definition we have $$|E(A,B)|\ge8\delta((t+1)km)^2.$$

\medskip
Every edge in $E(A,B)$

\smallskip
\noindent (i) belongs to $E(C_a,D_b)$ for some $\delta$-dense pair $(C_a,D_b)$, or

\noindent (ii) belongs to $E(C_a,D_b)$ for a pair $(C_a,D_b)$ that is not $\delta$-dense, or

\noindent (iii) has at least one endpoint outside $X=\bigcup_{a=1}^{tk}(C_a\cup D_a)$.

\smallskip
Note that each of the sets $A$ and $B$ has exactly $km$ elements outside $X$, so the number of edges that belong to the last category is less than $2km|A|=2(t+1)k^2m^2$. If $(C_a,D_b)$ is not $\delta$-dense, then $|E(C_a,D_b)|<\delta m^2$. Hence, the total number of edges of category (ii) is smaller than $(tk)^2\delta m^2$. Thus, the number of edges of category (i) is larger than
$$|E(A,B)|-2(t+1)k^2m^2-(tk)^2\delta m^2>4tk^2m^2.$$
Each $\delta$-dense pair $(C_a,D_b)$ contributes $|E(C_a,D_b)|\le m^2$ edges to this category. Therefore, the number of $\delta$-dense pairs $(C_a,D_b)$ is larger than $4tk^2$.

Consider now a pair of points $x,y\in A$. Clearly, if $x<_By$, then $x<_{D_b}y$  for all $1\le b\le tk$.

\begin{claim}\label{c1}
If points $x$ and $y$ in $A$ are incomparable in $<_B$, they are still comparable with respect to all but at most one ordering $<_{D_b}$, $1\le b\le tk$.
\end{claim}

\begin{proof}[Proof of Claim~\ref{c1}]
To verify this claim, suppose for a contradiction that $x$ and $y$ are incomparable in both $<_{D_a}$ and $<_{D_b}$ for some $1\le a<b\le tk$; see Figure \ref{Fig:Decomposition} (right). Hence, the line $xy$ must cross both $\conv(D_a)$ and $\conv(D_b)$.

On the other hand, since any two points $z\in D_a$ and $w\in D_b$ satisfy $z<_A w$, for every triple $x'\in A,z\in D_a$, and $w\in D_b,$ we have that $x'\in \ell(zw)$ and $w\in \ell(x'z)$.
It readily follows that $D_b$ is separated from $A\cup D_a$ by a common tangent to $\conv(A)$ and $\conv(D_a)$. In particular, $\conv(D_a)$ and $\conv(D_b)$ must be disjoint.

Assume with no loss of generality that the directed line $xy$ meets $\conv(D_a)$ and $\conv(D_b)$ in this order. Fix $z\in D_a\cap \ell(xy)$ and $w\in D_b\cap \ell(yx)$. Then the line $zw$ meets $xy$ between $\conv(D_a)$ and $\conv(D_b)$. Hence, both $x$ and $y$ must lie in the half-plane $\ell(wz)$, which contradicts the assumption that $A\in \ell(zw)$.
(If the line $xy$ meets $\conv(D_b)$ before $\conv(D_b)$, we apply a symmetric argument with $z\in D_a\cap \ell(yx)$ and $w\in D_b\cap \ell(xy)$.)
\end{proof}

It follows from Claim~\ref{c1} that
$$\sum_{1\le a,b\le tk}\iota(C_a,<_{D_b})\le\iota(A,<_B).$$
By symmetry, we also have
$$\sum_{1\le a,b\le tk}\iota(D_b,<_{C_a})\le\iota(B,<_A).$$
Thus, we obtain that
$$\sum_{1\le a,b\le tk}(\iota(C_a,<_{D_b})+\iota(D_b,<_{C_a}))\le\iota(A,<_B)+\iota(B,<_A)\le\eps\delta((t+1)km)^2,$$
where the last inequality follows from the assumption that $\{A,B\}$ is an $\eps\delta$-avoiding pair.
By definition, any pair $(C_a,D_b)$ that is {\em not} $\eps$-avoiding contributes more than $\eps m^2$ to this sum, so the number of such pairs is smaller than $\delta(t+1)^2k^2<2tk^2$.

All pairs $\{C_a,D_b\}$ are separated and consist of $m$-element sets. Call such a pair \emph{eligible} if it is both $\delta$-dense and $\eps$-avoiding. As we saw above, the number of $\delta$-dense pairs $(C_a,D_b)$ is larger than $4tk^2$. We have just seen that fewer than $2tk^2$ of them are not $\eps$-avoiding. Thus, the number of eligible pairs is larger than $2tk^2$.

Define a partial order on the set of eligible pairs as follows. Let $\{C_a,D_b\}<\{C_{a'},D_{b'}\}$ if $a<a'$ and $b<b'$. If there was no monotone chain of length $k$ with respect to this partial order, then by the dual of Dilworth's theorem \cite{Dilworth} all eligible pairs could be covered by fewer than $k$ antichains (i.e., fewer than $k$ sets of pairwise incomparable eligible pairs). The value $a-b$ is distinct for each pair $\{C_a,D_b\}$ in an antichain. Here we have $-kt<a-b<kt$, so antichains have fewer than $2kt$ eligible pairs. Hence, in this case the total number of eligible pairs would be smaller than $2tk^2$, contradicting our above estimate.

Thus, there exists a monotone chain of $k$ eligible pairs, say $\{A_1,B_1\}<\{A_2,B_2\}<\cdots<\{A_k,B_k\}$. These pairs obviously satisfy the requirements of Lemma \ref{4+}.
\end{proof}

By repeated application of Lemma~\ref{4+}, we obtain the following result.

\begin{lemma}\label7
Let $s$ and $u$ be positive integers such that $u$ is a multiple of $8^s$, and set $\delta=8^s/u$, $\eps=1/(32u^{s+2})$, and $K=(512u)^{s\choose2}$.

There exists a positive integer $M=M_s(u)\le u^sK$ such that for any two $M$-element sets of vertices, $A$ and $B$ that form a $\delta$-dense $\eps$-avoiding pair in a geometric graph $G$, the set of edges $E(A,B)$ connecting them contains $K$ pairwise crossing elements.
\end{lemma}

\begin{proof}
We prove the lemma by induction on $s$. The statement is trivial for $s=1$, as in this case we have $K=1$, so we can choose $M=1$ and select any one edge from $E(A,B)$ (which is nonempty, because $\delta>0$).

Let $s>1$ and assume that the statement holds for $s-1$ in place of $s$. We apply Lemma~\ref{4+} with $t=u/8^{s-1}$, $k=(512u)^{s-1}$, and $m=M_{s-1}(u)$. Setting $M=M_s(u)=(t+1)km$, the lemma states that given a $\delta$-dense $\eps$-avoiding pair $\{A,B\}$ with $|A|=|B|=M$, we can find pairwise disjoint $m$-element subsets $A_1,\dots,A_k\subset A$ and $B_1,\dots,B_k\subset B$  such that $(A_i,B_i)$ is a $1/t$-dense $1/(32t^2k)$-avoiding pair for every $i$, and every edge in $E(A_i,B_i)$ crosses all edges in $E(A_j,B_j)$, provided that $i\ne j$. For each $i$, we use the inductive hypothesis with the same $u$ and with $s-1$ in place of $s$ to find a family $Z_i$ of $K_0=(512u)^{s-1\choose2}$ pairwise crossing edges in $E(A_i,B_i)$. Note that $(A_i,B_i)$ is $1/(32t^2k)$-avoiding. To apply the induction hypothesis, we need that $(A_i,B_i)$ is $1/(32u^{s+1})$-avoiding, but  this holds, as $$1/(32t^2k)=1/(32\cdot8^{s-1}u^{s+1})\le1/(32u^{s+1}).$$ The union of the sets $Z_i$ contains $kK_0=K$ edges of $E(A,B)$, any two of which cross. Since $m=M_{s-1}(u)\le u^{s-1}K_0$, we obtain that
$$M=M_s(u)=(t+1)km\le uku^{s-1}K_0=u^sK.$$
\end{proof}

Part (ii) of Theorem~\ref{Thm:Main} follows by combining Lemmas~\ref6 and \ref7.

\begin{proof}[Proof of Theorem~\ref{Thm:Main}(ii)]
Let $G$ have $n$ vertices and at least $n^{2-x}$ edges with $x>(\log n)^{-2/3}$. We set the value of $s$ later and set $u=8^s\lceil n^x\rceil$. Apply Lemma~\ref6 with $\delta=8^s/u$, $\eps=1/(32u^{s+2})$ and $m=M_s(u)$ from Lemma~\ref7. The density condition of Lemma~\ref6 is satisfied, so if $n$ is large enough, we obtain separated vertex sets $A$ and $B$ of size $m$ forming a $\delta$-dense, $\eps$-avoiding pair. If this happens, we apply Lemma~\ref7 with the parameters $s$ and $u$ to obtain $K=(512u)^{s\choose2}$ pairwise crossing edges in $E(A,B)$.

It remains to do the calculation to check what value we can choose for $s$ and how many pairwise crossing edges we find this way. We can apply Lemma~\ref6 if $n\ge cm\delta^{-5}\eps^{-4}$ for the absolute constant $c$ in that lemma. Using $m=M_s(u)\le u^sK$ and the formulas defining $\delta$ and $\eps$, we see that $n\ge32^4cu^{5s+13}K$ suffices. In case $n=O(u^{5s+13}K)$, we have $n/K=O(u^{5s+13})$. Here $n\ge K\ge u^{s\choose2}\ge n^{x{s\choose2}}$, so $x{s\choose2}\le1$ and $s=O(1/\sqrt x)$. We also have $u=2^{O(s)}n^x=n^{x+O(s/\log n)}=n^{O(x)}$, since $s/\log n=O(x^{-1/2}/\log n)=O(x)$ by our lower bound on $x$. We have $n/K=u^{O(s)}=n^{O(sx)}=n^{O(\sqrt x)},$ as claimed.

In the above line of reasoning we assumed that the size of $G$ is just suitable, namely $n\ge32^4cu^{5s+13}K$, but barely. In the general case, we simply set the value of the parameter $s$ to be the largest possible value. That is, we still have $n\ge32^4cu^{5s+13}K$, but the same formula is violated if $s$ is increased by one. Note that increasing $s$ to $s^*=s+1$ changes the values of $u$ to $u^*=8u$ and $K$ to $K^*=512^s\cdot8^{s+1\choose2}u^sK$. We have $n<32^4c{u^*}^{5s^*+13}K^*=2^{O(s^2)}u^{O(s)}K$. We still have $n/K=2^{O(s^2)}u^{O(s)}=n^{O(xs+s^2/\log n)}=n^{O(\sqrt x)}$, as needed. Finally, if even the choice of $s=1$ is not feasible, then finding any one edge of the graph suffices.
\end{proof}

Note that a complete graph on $n$ vertices has $n^{2-x}$ edges for $x=\Theta(1/\log n)$. Thus, part (i) of Theorem~\ref{Thm:Main} would follow as a special case of part~(ii), if not for the $\eps>(\log n)^{-2/3}$ requirement there. Because of this, part~(ii) directly implies the existence of only $n/2^{O((\log n)^{2/3})}$ pairwise crossing segments determined by $n$ points in the plane in general position. To obtain the slightly stronger estimate claimed in part~(i) of Theorem~\ref{Thm:Main} we have to do the calculations again for the special case of the complete geometric graph. Here we do not have to worry about maintaining the density condition. We can apply Lemmas~\ref6 and \ref{4+} as stated, but we have to prove an analogue of Lemma~\ref7.

\begin{lemma}\label{complete}
Let $s$ be a positive integer and set $K=8^{s\choose2}$, $M=9^sK$, $\eps=2^{-3s-11}$. Suppose the $M$-element point sets $A$ and $B$ form a $\eps$-avoiding pair and $A\cup B$ is in general position. Then we can find $K$ pairwise crossing segments, each connecting a point of $A$ to a point of $B$.
\end{lemma}

\begin{proof} We prove the lemma by induction on $s$. For $s=1$ we have $K=1$, so the statement is trivial.

If $s>1$ we apply Lemma~\ref{4+} with $t=8$, $k=8^{s-1}$ and $m=M/(9k)$ to the complete geometric graph on $A\cup B$. As $A$ and $B$ form a $\left(2^{-14}/k\right)$-avoiding $1$-dense pair, the lemma finds us pairwise disjoint $m$-element subsets $A_i\subset A$ and $B_i\subset B$, for $1\leq i\leq k$, such that $\{A_i,B_i\}$ is $2^{-11}/k$-avoiding for all $i$ and each edge in $E(A_i,B_i)$ crosses all edges in $E(A_j,B_j)$ whenever $i\ne j$. We apply the inductive hypothesis for $s-1$ in place of $s$ separately for each pair $\{A_i,B_i\}$. This results in a subset $Z_i$ of $E(A_i,B_i)$ consisting of $K^*=8^{s-1\choose2}$ pairwise crossing edges. Their union, $\cup_{i=1}^kZ_i$ is a subset of $E(A,B)$ consisting of $kK^*=K$ pairwise crossing edges, as claimed.
\end{proof}

\begin{proof}[Proof of Theorem~\ref{Thm:Main}(i)]
We assume $n\ge3$.
Let us choose $s$ to be the smallest positive integer such that $V$ \emph{does not} determine $K=8^{s\choose2}$ pairwise crossing segments. Note that $s>1$, as $V$ determines at least $8^{1\choose2}=1$ pairwise crossing segments. By Lemma~\ref{complete}, we do not have size $M=9^sK$ subsets $A$ and $B$ of $V$ forming an $\eps$-avoiding pair with $\eps=2^{-3s-11}$. The complete geometric graph on the vertex set $V$ has $n$ vertices and $\delta n^2$ edges with $\delta=(n-1)/(2n)\ge1/3$. Applying Lemma~\ref6 to this graph yields $n=O(M\eps^{-4}\delta^{-5})=2^{O(s)}K$.

By the choice of $s$, $V$ determines at least $K^*=8^{s-1\choose2}=K/2^{O(s)}$ pairwise crossing segments. Therefore, $n>K^*=2^{s-1\choose2}$, so we must have $s=O(\sqrt{\log n})$. But we also have $K^*=K/2^{O(s)}=n/2^{O(s)}=n/2^{O(\sqrt{\log n})}$, as claimed.
\end{proof}

\section{Discussion and related results}\label{Sec:Discuss}

\begin{itemize}
\item Theorem \ref{Thm:Main} represents a substantial step towards characterizing the intersection structure of the edges in a geometric  graph. We hope that further progress in this direction would facilitate the solution other unsolved problems in combinatorial and computational geometry. The two most notorious questions of this kind are the following. Determine (1) the maximum number of halving lines of a set of $n$ points in $\reals^2$, and (2) the maximum number of incidences that can occur between $n$ points and $m$ pseudo-algebraic\footnote{We say that a family of curves in $\reals^2$ is pseudo-algebraic if any two of the curves intersect at most a fixed number of times. In particular, this includes families of bounded-degree algebraic curves.} curves in $\reals^2$. Note that the best known general upper bounds for these problems \cite{Dey,PachSharir} are obtained by applying the Crossing Lemma to a suitable geometric or topological graph. Following the pioneering work of Dvir~\cite{Dv10}, Guth and Katz~\cite{GK10, GK15}, several of these questions have been revisited from an algebraic perspective \cite{Zarank,KMS,ShZ}.

\item A closely related line of work \cite{KaLa,Kupitz,Pinchasi,Va3} concerns {\it pairwise evading edges} in geometric graphs. We say that a pair of segments in the plane are {\it evading} if (the closure of) neither of them is crossed by the supporting line of the other segment.
Note that the methods of Aronov {\em et al.}~\cite{CrossingFamilies} imply that any set of $n$ points in general position in the plane determines $\Omega\left(\sqrt{n}\right)$ pairwise evading segments. Analogously, our techniques yield the following result.


\begin{theorem}\label{Theorem:Parallel}
(i) Any set $V$ of $n$ points in general position in the plane determines at least $n/2^{O(\sqrt{\log n})}$ pairwise evading segments.

(ii) There exists an absolute constant $c$ such that any geometric graph with $n$ vertices and at least $n^{2-\eps}$ edges, for some $\eps>(\log n)^{-2/3}$, has at least $n^{1-c\sqrt\eps}$ pairwise evading edges.
\end{theorem}

Part (i) of Theorem \ref{Theorem:Parallel} easily follows from Theorem \ref{Thm:Main} via a reduction in \cite[Theorem 2]{CrossingFamilies}, while both parts can be derived using the following slight modification of Lemma \ref{4+}: For any separated $8\delta$-dense, $\eps\delta$-avoiding pair $(A,B)$ with $|A|=|B|=(t+1)km$, as in Lemma \ref{4+}, there exist $m$-element subsets $A_1,\ldots, A_k\subset A$, $B_1,\ldots, B_k\subset B$ so that for every $1\leq i\leq k$, the pair $\{A_i,B_i\}$ is $\delta$-dense and $\eps$-avoiding, and, for $1\leq i<j\leq k$, all edges in $E(A_i,B_i)$ evade every edge in $E(A_j,B_j)$.

To verify this statement, we closely follow the  proof of Lemma \ref{4+}. Specifically, we consider the partial orders $(A,<_B)$ and $(B,<_A)$, and use the following geometric property: if $x<_B y$ for $x,y\in A$ and $z<_A w$ for $z,w\in B$, then the segments $xw$ and $yz$ are evading, because $x,y,w$, and $z$ form a convex quadrilateral. Recall that Lemma \ref{4+} yielded $m$-size subsets  $C_1<_B C_2<_B\ldots <_B C_{tk}$ and $D_1<_A D_2<_A\ldots <_A D_{tk}$ of $A$ and $B$, respectively. Note that whenever $1\leq i<j\leq tk$ and $1\leq i'<j'\leq tk$, all edges in $E(C_i,D_{j'})$ evade every edge in $E(C_j, D_{i'})$; see Figure \ref{Fig:Parallel}. Now we apply Mirsky's theorem to the {\it modified} partial order, in which $(C_a,D_b)<(C_{a'},D_{b'})$ if and only if $a<a'$ and $b>b'$.

\begin{figure}[htbp]
\begin{center}
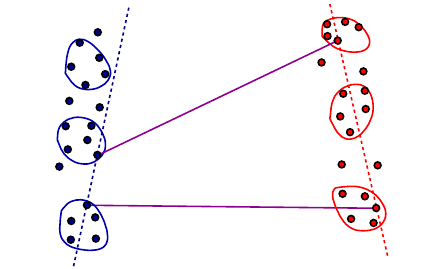
\caption{\small Proof of Theorem \ref{Theorem:Parallel}. All edges in $E(C_i,D_{j'})$ evade every edge in $E(C_j, D_{i'})$, whenever $C_i<_B C_j$ and $D_{i'}<_A D_{j'}$.}
\label{Fig:Parallel}
\end{center}
\end{figure}

\item Our proof of Theorem \ref{Thm:Main}, just like the arguments of Aronov et al.\ \cite{CrossingFamilies}, crucially relies on the assumption that the edges of our graph are straight-line segments. In particular, the decomposition provided by Lemma \ref{5} is based on the machinery that was developed in computational geometry to support efficient searching in arrangements of lines and hyperplanes. Extending Theorem \ref{Thm:Main} to more general families of topological graphs (e.g., whose edges are contained in fixed-degree algebraic curves, or any two of whose edges cross a bounded number of times) would require new ideas. Most of the previous lower bounds of this kind are based on Ramsey-type results applied to the intersection graph of the edge set \cite{Semi,FP08,FP10,FP12,FP12a,FP14,FPT11}.

For example, Fox and Pach \cite{FP10} argued that, with a suitable constant $c>0$, any $n$-vertex topological graph in the plane with at least $n\log^{c\log s}n$ edges must contain $s$  pairwise crossing edges all of whose vertices are distinct. In other words, using the terminology introduced in the Introduction, any $s$-quasi-planar topological graph has at most $n\log^{O(\log s)}n$ edges.
Their recursive analysis rests on the crucial observation that for any topological graph $G=(V,E)$ at least $\epsilon |E|^2$ crossing pairs of edges, there exist subsets $E_1,E_2\subseteq E$ of cardinality $t=\epsilon^{O(1)}|E|/\log n$ such that every edge in $E_1$ crosses every edge in $E_2$.
Furthermore, if we assume that every pair of edges in $E$ cross a bounded number of times (e.g., if they are bounded-degree semi-algebraic curves~\cite{Semi}), we can set $t=\epsilon^{O(1)} |E|$ \cite{FPT11}.
This is a special case of the following  general Tur\'{a}n-type phenomenon: any sufficiently dense string graph (in particular, the intersection graph of the edge set $E$) must contain a copy of $K_{t,t}$ \cite{FP12}.

It is possible that several of the above bounds can be improved by a careful divide-and-conquer scheme applied to the vertex set $V$, in the spirit of our proof of Theorem \ref{Thm:Main}.

\item Our proof of Theorem \ref{Thm:Main} is {\em fully constructive}. As the primary focus of this study is on the combinatorial aspects of geometric graphs, we did not seek to optimize the construction cost of our family of pairwise crossing edges. Nevertheless, our argument yields an algorithm whose running time is $O\left(n^{2+O(\sqrt{x})}\right)$ if the input graph $(V,E)$ has at least $n^{2-x}$ edges (for $x\geq (\log n)^{-2/3}$). Specifically, the $O\left(\frac{1}{\eps^2}\right)$-size line set $L$ of Lemma \ref{5} can be computed in time $O\left(\frac{1}{\eps^2}+\frac{n}{\eps}\right)$, using the deterministic algorithm of Chazelle \cite{Chazelle} for finding an optimal cutting of the dual plane. The $\eps$-avoiding $\delta$-dense pair $\{A,B\}$ (together with the posets $(A,<_B)$ and $(B,<_A)$) in Lemma \ref{6} can be computed, using a naive implementation, in $O\left(\frac{n^2}{\eps^4\delta^4}\right)$ time. The initial decomposition $\{C_a\}$ and $\{D_b\}$ of the sets $A$ and $B$ in Lemma \ref{4+} can be obtained (as designated in Lemma \ref{3}) via a simple topological sorting of the posets $(A,<_B)$ and $(B,<_A)$, and in time that is at most quadratic in $|A|=|B|$. The induced incomparability graphs $\iota(C_a,<_{D_b})$ (and, therefore, the eligible pairs $(C_a,D_b)$) can be determined in time $O\left((tk)\cdot(\eps\delta |A|^2)\right)=O(|A|^2)$. The desired longest chain of eligible pairs can be obtained in $O\left(t^4k^4\right)$ time (again, via the standard topological sorting of their poset \cite{Cormen}).
\end{itemize}

\paragraph{Acknowledgements.} An extended abstract of this paper has appeared in the Proceedings of the {51st Annual ACM Symposium Theory Computing (STOC 2019)}, pp. 1158--1166.
The authors would like to thank the anonymous referees for their valuable suggestions which helped to improve the presentation. In particular, we would like to thank one of the referees for bringing the matter of evading edges to our attention.

\end{document}